\documentclass{article}

\usepackage{amsfonts,amssymb,amsthm,amsmath}

\newtheorem{theorem}{Theorem}[section]
\newtheorem{lemma}[theorem]{Lemma}

\theoremstyle{definition}

\newtheorem{proposition}[theorem]{Proposition}
\newtheorem{corollary}[theorem]{Corollary}

\theoremstyle{remark}

\numberwithin{equation}{section}

\newcommand{\diffspec}{\operatorname{Spec^\Delta}}

\title{Global sections of structure sheaves of Keigher rings\footnote{This paper is dedicated to Jerald J. Kovacic}
}
\author{Dmitry Trushin\footnote{The author was partially supported by the grants: NSF
CCF-0964875 and 0952591} }
\date{}

\begin{document}

\maketitle

\begin{abstract}
Answering a question of J.~Kovacic, we show that, for any Keigher
ring, its differential spectrum coincides with the differential
spectrum of the ring of global sections of the structure sheaf. In
particular, we obtain the answer for Ritt algebras, that is,
differential rings containing the rational numbers.
\end{abstract}

\section*{Introduction}

The study of differential schemes began in~\cite{K1,K2,K3,K4,K5} and
was continued in~\cite{CF1,CF2} and~\cite{B1}. In~\cite{CF3}, a
different approach is taken than what we take here; in particular,
the definition of the structure sheaf is different. \cite{B2} has
yet another approach. In that work, a differential scheme is a
scheme whose structure sheaf consists of differential rings.
Further, the theory of differential schemes was developed
in~\cite{J,J1}.

We will use the definition of a structure sheaf $\mathcal O$ on
the differential spectrum $X$ of a differential ring $A$ as
in~\cite{J,J1}. It turns out that the ring of global sections of the
structure sheaf does not necessarily coincide with the initial ring.
Then, it is natural to ask whether the differential spectrum of the
ring of global sections $\mathcal O(X)$ coincides with the
differential spectrum of the initial differential ring $A$. Several
situations in which  the desired equality holds are described
in~\cite[Proposition~10.5]{J}. However, the most interesting case is
the case of Ritt algebras, that is differential rings containing
$\mathbb{Q}$. \cite[Example~7.4]{J1} shows that
\cite[Proposition~10.5]{J} cannot be applied to all Ritt algebras.
We resolve this issue for a wider class of differential rings -- the
Keigher rings~\cite[Definition~2.3]{J}.

The idea of the proof is the following. First, we modify the
structure sheaf to the one that is both easier for us to deal with
and does not lose the information about the differential spectrum.
Using simple topological fact, we show that the differential
spectrum of global sections of the new sheaf coincides with that of
the initial ring. In our method, the hypothesis that the ring is a
Keigher ring is used essentially.

The paper is organized as follows. In Section~\ref{sec:sec1}, we
recall some basic definitions taken from~\cite{J} and~\cite{J1}. In
Section~\ref{sec:sec2}, we construct the above mentioned auxiliary
sheaf and investigate its properties. The topological result is
presented in Lemma~\ref{top} and the main result of the section is
in Theorem~\ref{homeo}. In Section~\ref{sec:main}, we prove the main
result (Theorem~\ref{main}). Corollary~\ref{maincor} gives the
answer in the case of Ritt algebras.

\section{Terms and notation}\label{sec:sec1}

Throughout the text the word ring means an associative commutative
ring with a unit. All homomorphisms preserve the unit. A
differential ring is a ring with finitely many pairwise commuting
derivations. A differential ring is called a Keigher ring if for any
differential ideal $\frak a$ its radical $\frak r(\frak
a)=\left\{x\mid x^n\in\frak a\right\}$ is a differential ideal too.
For any differential ring $A$, the set of all prime differential
ideals is denoted by $\diffspec A$ and is called the differential
spectrum of $A$. For any subset $E\subseteq A$, the set of all prime
differential ideals containing $E$ is denoted by $V(E)$. Defining
$V(E)$ as closed subsets, we provide the differential spectrum with
the Kolchin topology. For any subset $E\subseteq A$, the smallest
radical differential ideal containing $E$ will be denoted by
$\{E\}$.

Denote the differential spectrum of $A$ by $X$ and let $s \in A$.
Then, the set of all prime differential ideals not containing $s$ is
denoted by $X_s$. For the set $S=\{s^n\}_{n=0}^\infty$, the
localization $S^{-1}A$ is denoted by $A_s$. If $\frak p$ is a prime
ideal then  we denote the ring $S^{-1}A$ by $A_\frak p$, where
$S=A\setminus \frak p$. For any differential homomorphism $f\colon
A\to B$, we define the corresponding map $$f^*\colon \diffspec B\to
\diffspec A,\quad f^*(\frak p)=f^{-1}(\frak p).$$

We will briefly recall the notion of structure sheaf on the differential
spectrum of a differential ring $A$. Consider the set
of all functions
$$
\left\{f\colon X\to \bigcup_{\frak p\in X} A_\frak p\mid f(\frak p)\in
A_\frak p\right\}.
$$
A function $f$ will be called regular at a point $\frak p$ if there
exists an open neighborhood $U$  of $\frak p$ and $a, b\in A$, where
for all $\frak q\in U$ we have $b\not\in\frak q$, such that for all
$\frak q\in U$ it follows that $f(\frak q)=a/b$ in $A_\frak q$. The
set of all functions that are regular at all points of the open
subset $U$ will be denoted by $\mathcal O(U)$. The family $\mathcal
O(U)$ with the restriction homomorphisms form a sheaf of
differential rings~\cite[Section~4]{J}. We will call this sheaf a
structure sheaf.

The ring $\mathcal O(X)$ will be denoted by $\widehat{A}$. There is
a map $\iota\colon A\to \widehat A$ such that $a\mapsto \varphi$,
where $\varphi(\frak p)=a/1$ in $A_\frak p$ One can show that
$\iota$ is a differential homomorphism. Further details about
the structure sheaf can be found in~\cite{J}.

\section{Auxiliary sheaf}\label{sec:sec2}

In this section, we construct another sheaf for which the desired
problem will be solved by a straightforward calculation. The general
proof given in Section~3 is based on a reduction of the initial
problem to the problem solved in this section.

Let $A$ be an arbitrary differential ring. For any prime
differential ideal $\frak p$, its residue field will be denoted by
$K(\frak p)$, that is, the fraction field of $A/\frak p$. Consider the
set of all functions
$$
\left\{f\colon X\to \bigcup_{\frak p\in X}K(\frak p)\mid f(\frak p)\in
K(\frak p)\right\}.$$
A function $f$ will be called regular at a point $\frak p$ if there
exists an open neighborhood $U$ of $\frak p$ and  $a,b\in A$,
where for all $\frak q\in U$ we have $b\not\in \frak q$, such that
for all $\frak q\in U$ it follows that $f(\frak q)=a/b$ in $K(\frak
q)$. The set of all functions that are regular at all points of $U$ will
be denoted by $\mathcal O'(U)$. The family of all such rings with
the restriction homomorphisms form a sheaf of differential rings on $X$.

Define a differential homomorphism $\iota'\colon A\to \mathcal
O'(X)$ by the rule $a\mapsto \varphi$ such that $\varphi(\frak p)=a$
in $K(\frak p)$. Our proof is based on the following statement.

\begin{lemma}\label{top}
Let $\varphi\colon X\to Y$ be a map of topological spaces and let
two covers $U_\alpha$ and $V_\alpha$ of $X$ and $Y$, respectively,
be given. Then, if for any pair of indices $\alpha$ and $\alpha'$
the map $\varphi\colon U_{\alpha}\cap U_{\alpha'}\to V_{\alpha}\cap
V_{\alpha'}$ is well-defined and is a homeomorphism then
$\varphi\colon X\to Y$ is a homeomorphism too.
\end{lemma}
\begin{proof}

Take $\alpha=\alpha'$. Then, $\varphi$ is a local
homeomorphism. We need to prove that $\varphi$ is injective. Let
$x\in U_\alpha$ and $y\in U_{\alpha'}$ be such that $z=f(x)=f(y)$.
Then, $z\in V_{\alpha}\cap V_{\alpha'}$. Since $\varphi$ is a
homeomorphism of $U_{\alpha}\cap U_{\alpha'}$ and
$V_{\alpha}\cap V_{\alpha'}$,  the intersection $U_{\alpha}\cap
U_{\alpha'}$ is not empty and contains $w$ such that $\varphi(w)=z$.
But $x$ and $w$ both belong to $U_{\alpha}$ and their images
coincide. Thus, $x=w$. A similar argument shows that $y=w$.
\end{proof}

Now, we will prove some auxiliary facts.

\begin{lemma}\label{lem:22}
Let $A$ and $B$ be differential rings,  $\nu\colon A\to B$ be a
differential homomorphism, and $B$ be differentially generated over
$A$ by a single element $\eta$ such that there exists a family of
elements $b_1,\ldots,b_n$ of $A$ with conditions $b_k \eta\in
\nu(A)$ and $\{b_1,\ldots,b_n\}=A$. Then, the contraction map
$$\nu^*\colon \diffspec B\to  V(\ker \nu)$$ is a homeomorphism.
\end{lemma}
\begin{proof}
First, replacing $A$ by $A/\ker\nu$, we may suppose that $\nu$ is
injective. Let $X=\diffspec A$ and $Y=\diffspec B$. Now, we will
divide the differential spectra into finitely many open subsets as
follows:
\begin{gather*}
X=X_{b_1}\cup\ldots\cup X_{b_n}\\
Y=Y_{b_1}\cup\ldots\cup Y_{b_n}
\end{gather*}
We need to check the hypothesis of the previous lemma for the
covers $\{X_{b_i}\}$ and $\{Y_{b_i}\}$. The proof is based
on~\cite[Exercise.~21 (i,ii) p.~46]{AM}. We replace $X_{b_i}$ by the
differential spectrum of $A_{b_i}$ and  $Y_{b_i}$ by the
differential spectrum of $B_{b_i}$. From the mentioned exercise, it
follows that the restriction of $\nu^*$ coincides with $\nu^*_{b_i}$
(see~\cite[Excercise~21 (ii) p.~47]{AM}).

In order to show the desired result, we will calculate both
localizations $A_{b_i}$ and $B_{b_i}$. Let $b_i\eta=a_i\in A$.
Consider
\[
B_{b_i}=(A\{\eta\})_{b_i}=A\{a_i/b_i,1/b_i\}=A\{1/b_i\}=A_{b_i}
\]
Therefore, the localization of  $\nu\colon A\to B$ by $b_i$ induces
the identity map
\[
\nu_{b_i}=Id\colon A_{b_i}\to B_{b_i}=A_{b_i}.
\]
So, the desired homeomorphism between $X_{b_i}$ and $Y_{b_i}$ has
been obtained.

The last step of the proof is to show that $\nu^*$ induces a
homeomorphism between $X_{b_i}\cap X_{b_j}$ and $Y_{b_i}\cap
Y_{b_j}$. Having noted that $X_{b_i}\cap X_{b_j}=X_{b_ib_j}$ and
$Y_{b_i}\cap Y_{b_j}=Y_{b_ib_j}$ and applying the previous argument
with $b_ib_j$ instead of $b_i$, we obtain the result.
\end{proof}

Using induction, we show the following.

\begin{lemma}\label{steplemma}
Let $A$ and $B$ be differential rings, $\nu\colon A\to B$ be a
differential homomorphism, and $B$ be differentially finitely
generated over $A$ such that for any generator $\eta$ there exist $b_1,\ldots,b_n$ satisfying $b_k \eta\in
\nu(A)$ and $\{b_1,\ldots,b_n\}=A$. Then, $$\nu^*\colon \diffspec B\to
V(\ker \nu)$$ is a homeomorphism.
\end{lemma}
\begin{proof}

Let $B$ be differentially generated over $A$ by
$\eta_1,\ldots,\eta_n$, and $B_1$ be the subring of $B$
differentially generated over $A$ by $\eta_1,\ldots, \eta_{n-1}$.
Then, the homomorphism $\nu$ can be represented as the following
composition
\[
A\stackrel{\nu_1}{\longrightarrow}
B_1\stackrel{\nu_2}{\longrightarrow} B.
\]
By Lemma~\ref{lem:22}, $V(\ker\nu_2)$ is homeomorphic to $\diffspec
B$. By the inductive assumption, $\diffspec B_1$ and $V(\ker \nu_1)$
are homeomorphic, and, thus, their corresponding subspaces $V(\ker
\nu_2)$ and $V(\ker\nu)$ are homeomorphic too. Consequently,
$V(\ker\nu)$ and $\diffspec B$ are homeomorphic as well.
\end{proof}

\begin{lemma}\label{Limitlemma}
Let $A_\alpha$ be a direct system of differential $A$-algebras and
$\nu_\alpha\colon A\to A_\alpha$ be the corresponding homomorphisms.
Suppose that for any $\alpha$ the map $\nu^*_\alpha$ is a
homeomorphism. Then, $\nu^*$ is a homeomorphism, where $\nu\colon
A\to \varinjlim\limits_{\alpha} A_\alpha $ is the direct limit of
$\nu_\alpha$.
\end{lemma}
\begin{proof}

Denote $\varinjlim\limits_{\alpha} A_\alpha$  by $A'$. We will first
show that $\nu$ is bijective. Surjectivity: consider a prime
differential ideal $\frak p$ in $A$. Let $\frak p_\alpha$ be the
ideals in $A_\alpha$ corresponding to $\frak p$ that is
$\nu_\alpha^*(\frak p_\alpha)=\frak p$. Then,
$\varinjlim\limits_{\alpha} \frak p_\alpha$ is a prime differential
ideal in $A'$ contracting to $\frak p$.


Injectivity: suppose that there exist two prime differential ideals
$\frak q$ and $\frak q'$  in $A'$ contracting to $\frak p$. Then,
there exists $s\in A'$ such that $s\in \frak q'\setminus\frak
q$. It follows from \cite[Chapter~2, Exercise~14]{AM} that there are
$\alpha$ and $s_\alpha \in A_\alpha$ such that $s_\alpha$
maps to $s$. Let $\frak q_\alpha$ and $\frak q'_\alpha$ denote
the contractions of $\frak q$ and $\frak q'$ on $A_\alpha$, respectively.
Then, we have $s_\alpha\in\frak q'_\alpha\setminus\frak q_\alpha$.
But both  ideals contract to $\frak p$ and belong to
$A_\alpha$, contradiction.

We will show now that $\nu$ is a homeomorphism. We just need to
prove that the image of any principal open set is open. Let $s$ be
an arbitrary element of $A'$. We will demonstrate that the image of
$\diffspec (A')_s$  is open in $\diffspec A$. Indeed, for some
$\alpha$, there is $s_\alpha\in A_\alpha$ such that its
image coincides with $s$. Then, $$\diffspec
(A_\alpha)_{s_\alpha}=\diffspec (A')_s.$$ But, by the data, the image
of $\diffspec (A_\alpha)_{s_\alpha}$ is open in $\diffspec A$,
which finishes the proof.
\end{proof}

\begin{lemma}\label{lem:25}
For every $\eta\in \mathcal O'(X)$, there exist $b_1,\ldots,b_m\in A$ such that
$b_i\eta\in\iota'(A)$ and $\{b_1,\ldots,b_m\}=A$.
\end{lemma}
\begin{proof}
Without the loss of generality, we may suppose that
$U_1\cup\ldots\cup U_m=X$ such that $U_i=X_{b_i}$ and, for any
$\frak p\in U_i$ we have $\eta(\frak p)=a_i/b_i$ in $K(\frak p)$. We
will replace  $a_i$ and $b_i$ by $a_ib_i$ and $b_i^2$, respectively.
Then, for any $\frak p\not\in U_i$, we have $a_i=0$ in $K(\frak p)$.
Then, since all $U_i$ cover $X$, it follows that
$\{b_1,\ldots,b_m\}=A$. Let us show that, for every $\frak p$,
$b_i\eta(\frak p)=a_i$ in $K(\frak p)$ holds. Indeed, if $\frak p\in
U_i$ then $$b_i\eta(\frak p)=b_i(a_i/b_i)=a_i$$ in $K(\frak p)$. If
$\frak p\in U_j\setminus U_i$ then $$b_i\eta(\frak
p)=b_i(a_j/b_j)=0$$ in $K(\frak p)$, but $a_i=0$ in $K(\frak p)$
too. Consequently, $b_i\eta=\iota'(a_i)$.
\end{proof}

\begin{theorem}\label{homeo}
Let $D$ be a differential subring of $\mathcal O'(X)$ containing
the image of $A$ and
$\iota'\colon A\to D$ be the corresponding homomorphism. In this case,
\[
\iota'^*\colon \diffspec D\to \diffspec A.
\]
is a homeomorphism.
\end{theorem}
\begin{proof}

It follows from Lemma~\ref{lem:25} that for any differentially
finitely generated over $A$ subalgebra in $\mathcal O'(X)$ the
hypotheses of Lemma~\ref{steplemma} hold. From the definition of
$\iota'$, it follows that $\ker \iota'$ belongs to the intersection
of all prime differential ideals of $A$. Consequently, for any
differentially finitely generated algebra $D$ the statement is
proven. But any differential algebra can by presented as a direct
limit of its differentially finitely generated subalgebras. Then, it
follows from Lemma~\ref{Limitlemma}  that the statement holds for
any subalgebra $D$.
\end{proof}

\section{Main result}\label{sec:main}

In this section, we reduce the main result to Theorem~\ref{homeo}. We
have two sheaves on the topological space $X$, namely: $\mathcal O$
and $\mathcal O'$. Consider the homomorphism of sheaves
$$\varphi(U)\colon \mathcal O(U)\to\mathcal O'(U)\quad
\varphi(f)(\frak p)=f(\frak p)\mod \frak p.$$ From now we suppose
that the given ring $A$ is a Keigher ring.

\begin{proposition}\label{prop:31}
With the above conventions,  $\ker\varphi(X)$ coincides
with the nilradical of $\mathcal O(X)$.
\end{proposition}
\begin{proof}
By definition, $\ker\varphi(X)$ coincides with
\[
\{\,f\in \mathcal O(X)\mid f(\frak p)\in\frak pA_\frak p\,\}.
\]
Let us show that the nilradical is contained in the kernel. Let $f$
be a nilpotent element. Then, for any $\frak p$ we have $f^n(\frak
p)=0$. So, $f^n(\frak p)\in \frak pA_\frak p$. Consequently,
$f(\frak p)\in \frak pA_\frak p$.

Conversely, let $f\in \ker\varphi(X)$. Then, there exists a cover
$U_1\cup\ldots\cup U_n=X$ such that for every $i$ we have $f(\frak
p)=a_i/b_i$ in $A_\frak p$, whenever $\frak p\in U_i$. Without loss
of generality we may suppose that $U_i=X_{b_i}$. Fix some $i$. Then,
$$f(\frak p)=a_i/b_i\in \frak pA_\frak p,$$ for all $\frak p$ not
containing $b_i$. Then, $a_i\in\frak p$ for all $\frak p$ not
containing $b_i$. Therefore, $a_i$ belongs to the intersection of
all prime ideals in $A_{b_i}$. Since $A$ is a Keigher ring
($A_{b_i}$ too), $a_i$ belongs to the nilradical of
$A_{b_i}$~\cite[Corollary~2.9]{J}. Hence, for some $k_i$ we have
$a_i^{k_i}=0$ in $A_{b_i}$. Therefore, for some $m_i$, it follows
that $$b_i^{m_i}a_i^{k_i}=0$$ in $A$. And so, we have that
$(a_ib_i)^{n_i}=0$ for some $n_i$. Let $n=\max\limits_i n_i$. Then,
for every $i$ we have $(a_ib_i)^n=0$ in $A$. Consequently, for every
$i$ and every $\frak p\in U_i$ the following holds:
\[
f^n(\frak
p)=\left(\frac{a_i}{b_i}\right)^n=\frac{(a_ib_i)^n}{b_i^{2n}}=0\mbox{
â } A_{\frak p}.
\]
\end{proof}

\begin{theorem}\label{main}
Let $D$ be a differential subring of $\mathcal O(X)$ containing the
image of $A$ and $\iota\colon A\to D$ be the corresponding
homomorphism. Then,
\[
\iota^*\colon \diffspec D\to \diffspec A.
\]
is a homeomorphism.
\end{theorem}
\begin{proof}
Consider the homomorphism $\varphi(X)\colon D\to\mathcal O'(X)$. It
follows from Proposition~\ref{prop:31} that the kernel of
$\varphi(X)$ coincides with the nilradical of $D$. The quotient of
$D$ by its nilradical will be denoted by $D'$. It is clear that
$\diffspec D$ can be identified with $\diffspec D'$. We have the
following sequence of homomorphisms
\[
A\to D'\to \mathcal O'(X).
\]
And the composition coincides with $\iota'$. It follows from
Theorem~\ref{homeo}  that $\diffspec D'\to \diffspec A$
is a homeomorphism. Remembering that $\diffspec D$ is homeomorphic
to $\diffspec D'$, we obtain the desired result.
\end{proof}

\begin{corollary}\label{maincor}
Let $A$ be a Ritt algebra and $\iota\colon A\to\widehat{A}$ be the
canonical homomorphism into the ring of global sections. Then,
\[
\iota^*\colon\diffspec \widehat{A}\to \diffspec A
\]
is a homeomorphism.
\end{corollary}

\end{document}